\documentclass[11pt, reqno]{article}

\usepackage{geometry}                
\usepackage{graphicx}
\usepackage{verbatim} 
\usepackage{amssymb}
\usepackage{epstopdf}
\usepackage{amsmath,amsthm,amsfonts}
\usepackage{amssymb}
\usepackage{epsfig}
\usepackage{setspace}
\usepackage{color}

\newtheorem{theorem}{Theorem}[section]
\newtheorem{lemma}[theorem]{Lemma}
\newtheorem{proposition}[theorem]{Proposition}

\newtheorem{obs}[theorem]{Observation}

\newtheorem{corollary}[theorem]{Corollary}

\def\nn{{\mathbb N}}
\newcommand{\si}{\sigma}

\title{On the number of even roots of permutations}
\author{Lev Glebsky, Melany Lic\'on, Luis Manuel Rivera}
\date{}                                           

\begin{document}
\maketitle

\begin{abstract}
Let $\sigma$ be a permutation on $n$ letters. We say that a permutation $\tau$ is an even (resp. odd) $k$th root of $\sigma$ if $\tau^k=\sigma$ and $\tau$ is an even (resp. odd) permutation. In this article, we obtain generating functions for the number of even and odd $k$th roots of a permutation, in terms of its cycle type. Our result implies known generating functions of Moser and Wyman and also some generating functions for sequences in The On-line Encyclopedia of Integer Sequences (OEIS).
\end{abstract}

{\it Keywords:}  Roots of permutations; even permutations; generating functions.

{\it AMS Subject Classification Numbers:} 05A05; 05A15.

\section{Introduction}

 A classical problem in group theory and combinatorics is the study of problems related to the solution of the equation $x^k=a$ over groups, where $k$ is a fixed positive integer (see, e.g.,~\cite{chigi, lucidopour, lucidopour2, mw, sade, sade2, york}). One of the most studied situations is the case of the symmetric group $S_n$. For example, there is a  characterization that determines when a given permutation has a $k$th root in $S_n$ (see, e.g.,~\cite{annin, boucher, chs}) and there are several results about the probability that a randomly selected permutation of length $n$ has a $k$th root  (see, e.g.,~\cite{bona, cher, mw, niem, pouyanne}). In additon, Pavlov \cite{pavlov1} gave an explicit formula for the number of solutions in $S_n$ of the equation $x^k=\sigma$, and Lea\~nos et al., \cite{lmr} gave a multivariable exponential generating function. Finally, Roichman \cite{roichman} gave a formula for such a number expressed as an alternating sum of $\mu$-unimodal $k$th roots of the identity permutation.   
 
In this article, we are interested in the number of even permutations as the $k$th roots of a given permutation. To our knowledge\footnote{After finishing this work, we were aware of the existence of a generating function, given in terms of the cycle index of the symmetric group $S_n$ due to Chernoff \cite{cher2}.}, there are only few results in this direction and only for the case of the identity permutation. Moser and Wyman~\cite{mw} studied the case of $k=2$. In OEIS~\cite{oeis} there are only a few sequences for the number of even $k$th roots of identity permutation: A000704 ($k=2$), A061129 ($k=4)$, A061130 ($k=6$), A061131 ($k=8$) and A061132 ($k=10$). For the odd $k$th roots of the identity permutation, in OEIS we find sequences A001465 ($k=2$), A061136 ($k=4$) and A061137 ($k=6$). 


\subsection{Basic definitions and main result}

In order to formulate our main result, we need some definitions and notation. The {\it cycle type} of an $n$-permutation is a vector ${\bf c}=(c_1, \dots, c_n)$, which means that for every $i \in [n]$, the permutation has $c_i$ cycles of length $i$. We say that a permutation $\sigma$ is of {\it cycle type} $(\ell_1)^{a_1} \dots (\ell_m)^{a_m}$, with $a_i>0$, if $\si$ has exactly $a_i$ cycles of length $\ell_i$ in its disjoint cycle factorization and does not have any cycles of any other length. We use $\nn$ (respectively $\nn_0$) to denote the set of positive (respectively, non-negative) integers. Let $k, \ell \in \nn$.  Let

\[
	G_k(\ell)  = \{ g \ \in \nn : \gcd(g\ell, k) = g  \}.
\]
It is easy to see that if $k=p_1^{a_1} \cdots p_j^{a_j}$, where $p_1, \dots, p_j$ are distinct primes and $a_i >0$ for $i \in \{1, \dots, j\}$, then  
\[
	G_k(\ell)  = \left\{ p_1^{b_1} \cdots p_j^{b_j} : b_i=a_i \text{ if } p_i  | \ell \text{ and } b_i \in \{0, 1, \dots, a_i\} \text{ if } p_i \not | \ell   \right\}.
\]

The main result of this paper is the following. 
\begin{theorem}\label{maintheorem} Let $k, n$ be positive integer. Let $c_1,\dots,c_n$ be non-negative integers such that $n=c_1+2c_2+\cdots +nc_n$. Then the coefficient of $\frac{t_1^{c_1} \cdots t_{n}^{c_n}}{c_1! \cdots c_n!}$ in the expansion of
\begin{eqnarray*}
\frac{1}{2}\exp \left(\sum_{\ell \geq 1} \sum_{g \in G_k(\ell)} \frac{\ell^{g-1}}{g}t^g_{\ell} \right)+\frac{1}{2}{\rm exp}\left(\sum_{\ell \geq 1} \sum_{g \in G_k(\ell)}(-1)^{\ell g+1}\frac{\ell^{g-1}}{g}t_\ell^g\right)
\end{eqnarray*}
is the number of even $k$th roots of a permutation of cycle type ${\bf c}=(c_1, \dots, c_n)$, and in the expansion of 
\begin{eqnarray*}
\frac{1}{2}\exp \left(\sum_{\ell \geq 1} \sum_{g \in G_k(\ell)} \frac{\ell^{g-1}}{g}t^g_{\ell} \right)-\frac{1}{2}{\rm exp}\left(\sum_{\ell \geq 1} \sum_{g \in G_k(\ell)}(-1)^{\ell g+1}\frac{\ell^{g-1}}{g}t_\ell^g\right).
\end{eqnarray*}
is the number of odd $k$th roots of a permutation of cycle type  ${\bf c}$.
\end{theorem}

The known result of the identity permutation is a consequence of this theorem. The outline of this paper is as follows. In Section~\ref{proofs}, we will prove several propositions and lemmas that we use in the proof of our main result. The proof of Theorem~\ref{maintheorem} is at the end of this section. In Section~\ref{cases}, we show a few special cases of Theorem~\ref{maintheorem}, which allow some nice simplifications.

\section{Auxiliary results and proof of Theorem~\ref{maintheorem}}\label{proofs}

First, we present two known results, which will be used in the proof or our main result. 

\begin{proposition}[{\cite[Proposition 5]{lmr}}]\label{coro-sol} A permutation of cycle type $(\ell)^c$ has a $kth$ root if and only if the equation
			\[
		g_1 x_1 + \dots + g_h x_h = c
		\]
	has non-negative integer solutions, where $G_k(\ell)=\{ g_1, \dots, g_h\}$.
	\end{proposition}

The following result shows a generating function for the number of $k$th roots of a permutation.

\begin{theorem}[{\cite[Theorem 2]{lmr}}]\label{theolmr} Let $k, n$ be positive integer. Let $c_1,\dots,c_n$ be non-negative integers such that $n=c_1+2c_2+\cdots +nc_n$. Then the coefficient of $\frac{t_1^{c_1} \cdots t_{n}^{c_n}}{c_1! \cdots c_n!}$ in the expansion of
\begin{equation*}
\exp\left(\sum_{\ell\geq1}\sum_{g \in G_m(\ell)}\frac{\ell^{g-1}}{g}t_\ell^g\right)
\end{equation*}
is the number of $k$th roots of a permutation of cycle type ${\bf c}=(c_1,\dots,c_n)$.
\end{theorem}

The outline of the proof is as follows.  First, we work   with the difference between the number of even $k$th roots and the number of odd $k$th roots of a permutation (Lemma~\ref{dife-gene}).  The next step was to obtain a multivariable exponential generating function for such a difference (Lemma~\ref{egfdifference}). In order to do this, first we assign a sign to the number of $k$th roots, of certain type, of permutations with all its cycles of the same length (Proposition~\ref{signo-pari}). Using this, we obtain an exponential generating function for the difference between the number of even $k$th roots and odd $k$th roots of permutations with all its cycles of the same length (Lemma~\ref{fungen-lfijo}). Finally, the proof of Theorem~\ref{maintheorem}  is obtained as a consequence of Theorem~\ref{theolmr} and Lemma~\ref{egfdifference}.

We need the following easy proposition about groups in general.
\begin{proposition}\label{pr_diferen2}
  Let $G$ be a group and $K$ be a field. Let $\phi \colon G\to (K, \cdot)$ ($g\mapsto \phi^g$) be a homomorphism to the multiplicative group of $K$ and
  $X, Y\subseteq G$ be finite. Then
 \[
\left(\sum_{g\in X}\phi^g\right)\left(\sum_{h\in Y}\phi^h\right)=\sum_{\substack{g\in X \\ h\in Y}} \phi^{gh}.
 \]
\end{proposition}
Let $\operatorname{re}_k(\sigma)$ (resp. $\operatorname{ro}_k(\sigma)$) denote the number of even (resp. odd) $k$th
 roots  of permutation $\sigma$. The {\it support} of an $n$-permutation $\si$ is defined as ${\rm supp}(\si)=\{a \in \{1, \dots, n\}~\colon \si(a)\neq a\}$.

\begin{proposition}\label{cor_diferen2}
Let $\si$ be a permutation such that $\sigma=\si_1\si_2$ and ${\rm supp}(\sigma_1)\cap {\rm supp}(\sigma_2)=\emptyset$. Let $\operatorname{re}_k'(\si)$ (resp.  $\operatorname{ro}_k'(\si)$) be the number of even  (resp. odd)  $k$th roots $\tau$ of $\sigma$ such that $\tau=\tau_1\tau_2$ with $\tau_1^k=\si_1$ and $\tau_2^k=\si_2$. Then
\[
\operatorname{re}_k'(\si)- \operatorname{ro}_k'(\si)=\left(\operatorname{re}_k(\sigma_1)-\operatorname{ro}_k(\sigma_1)\right)\left(\operatorname{re}_k(\sigma_2)-\operatorname{ro}_k(\sigma_2)\right).
\]
\end{proposition}
\begin{proof}
Consider the parity of permutations  as a homomorphism $\phi \colon S_n \to \{-1,1\}$. Let  $X=\{\tau_1 \in S_n \colon \tau_1^k=\sigma_1\}$ and $Y=\{\tau_2 \in S_n \colon \tau_2^k=\sigma_2\}$. Then $
\sum_{\tau_1 \in X}\phi^{\tau_1}=\operatorname{re}_k(\si_1)- \operatorname{ro}_k(\si_1)$ and $ \sum_{\tau_2 \in Y}\phi^{\tau_2}=\operatorname{re}_k(\si_2)- \operatorname{ro}_k(\si_2)$. Therefore, by Proposition~\ref{pr_diferen2} we have that
\[
\sum_{\substack{\tau_1 \in X\\
\tau_2 \in Y}}\phi^{\tau_1\tau_2}=re'_k(\si)- ro'_k(\si).
\qedhere 
\]
\end{proof}

The following result shows that for a given permutation $\sigma$ we can obtain the difference $\operatorname{re}_k(\si)-\operatorname{ro}_k(\si)$ by working with the different lengths in the cycles of $\si$ separately.
  
\begin{lemma}\label{dife-gene} Let $\sigma$ be an $n$-permutation that has $k$th roots. Suppose that the disjoint cycle factorization of $\si$ can be expressed as the product $\si_1\si_2 \cdots \si_m$ where $\si_i$ is the product of all the disjoint cycles of length $\ell_i$ in $\si$, for every $i$, with   $\ell_i \neq \ell_j$, for $i \neq j$. Then
\[
\operatorname{re}_k(\si)-\operatorname{ro}_k(\si)=\prod_{i=1}^{m}\left(\operatorname{re}_k(\si_i)-\operatorname{ro}_k(\si_i)\right).
\]
\end{lemma}
\begin{proof}
It is well-known that every $k$th root of $\si$ can be written as $\tau_1\cdots \tau_m$ with $\tau_i^k=\sigma_i$, for every $i$ (see, e.g., \cite[\S 3]{lmr}). The result follows by Proposition~\ref{cor_diferen2} and induction.
 \end{proof}


Sometimes, we use the following fact: if $\alpha$ is an $\ell$-cycle, then $\alpha^m$ is a product of exactly ${\rm gcd}(m, \ell)$  disjoint $\ell/{\rm gcd}(m, \ell)$-cycles. Let $g, k, \ell$ be fixed positive integers and $p$ be a fixed non-negative integer. We use $f_{k, \ell, g, p}(c)$ to denote the number of permutations of cycle type $(g\ell)^{p}$ that are $k$th roots of a permutation of cycle type $(\ell)^{c}$,  $c \in \nn_0$. The following proposition has been proven, in essence, by Moreno et al.~\cite{lmr}.
\begin{proposition}\label{proppegado}
	Let $g, k, \ell$ be fixed positive integers and $p$ be a fixed non-negative integer. Let $c \in \nn_0$. If $g \in G_k(\ell)$ and $c=gp$, then  	
\[
f_{k, \ell, g, p}(c)=\frac{(gp)! \ell^{p(g-1)}}{g^pp!},
\]
and $f_{k, \ell, g, p}(c)=0$ in any other case.
\end{proposition}

In view of previous proposition, for $g \in G_k(\ell)$ we define
\[
f_{k, \ell, g}(c) = \left\{
        \begin{array}{ll}
            f_{k, \ell, g, p}(c) & \quad c=gp \\
            0 & \quad \text{ other case}
        \end{array}
    \right.
    \]
Now, we assign a sign to the number $f_{k, \ell, g}(c)$, which helps to know whether the roots of cycle type $(g\ell)^p$ of a permutation of cycle type $(\ell)^c$ are even.

\begin{proposition}\label{signo-pari}
 Let $k, \ell$ be fixed positive integers. Let $g \in G_k(\ell), c \in \nn_0$ and  
 \[
a(c)=(-1)^{c(\ell g+1)/g}f_{k, \ell, g}(c).
\]
If $\sigma$ is a permutation of cycle type $(\ell)^{c}$ and $c=gp$, then $a(c) \neq 0$. In addition,  the $k$th roots of cycle type $(g\ell)^{p}$ of $\sigma$ are even permutations if and only if $a(c)>0$.
\end{proposition}
\begin{proof}

As $c=gp$, we have that $a(c)=(-1)^{p(\ell g+1)}f_{k, \ell, g, p}(c)$, and Proposition~\ref{proppegado} implies that $a(c) \neq 0$. The result follows because the sign of a $q$-cycle is $(-1)^{q + 1}$ and hence the sign of the product of $p$ cycles of length $(\ell g)$ is $(-1)^{p(\ell g + 1)}$. 
\end{proof}
The exponential generating function, in the variable $t_\ell$, for the number $a(c)$ in the previous proposition is given in the following result.  
\begin{proposition}\label{propgenmenos}
	Let $\ell, k \in \nn$. Let $g \in G_k(\ell)$ fixed. Then
	\[
	\sum_{c \geq 0 } (-1)^{c/g(\ell g+1)}f_{k, \ell, g}(c)\frac{t_{\ell}^{c}}{c!}=\exp \left((-1)^{\ell g +1}\frac{\ell^{(g-1)}}{g}t^g_{\ell} \right).
	\]
\end{proposition}
\begin{proof}
From Proposition~\ref{proppegado} we have that $f_{k, \ell, g}(c) \neq 0$ if and only if $c = gp$, for some $p \in \nn_0$. Therefore  
\begin{eqnarray*}
\sum_{c \geq 0 } (-1)^{c/g(\ell g+1)}f_{k, \ell, g}(c)\frac{t_{\ell}^{c}}{c!}
&=& \sum_{p \geq 0 } (-1)^{p(\ell g+1)}	\frac{(gp)! \ell^{p(g-1)}}{g^pp!}  \frac{t_{\ell}^{gp}}{(gp)!}\\
&=& \sum_{p \geq 0 } \left((-1)^{(\ell g+1)}	\frac{ \ell^{(g-1)}}{g} t_{\ell}^{g}\right)^p \frac{1}{(p)!}\\
&=&\exp \left((-1)^{\ell g +1}\frac{\ell^{(g-1)}}{g}t^g_{\ell} \right).	
\qedhere 
\end{eqnarray*}

\end{proof}

Let $\operatorname{re}_k(\ell, c)$ (resp. $\operatorname{ro}_k(\ell, c)$) denote the number of even (resp. odd) $k$th roots of any permutation of cycle type $(\ell)^{c}$. 

\begin{lemma}\label{fungen-lfijo} Let $\ell \in \nn$. Then
\[
\sum_{c \geq 0} \left(\operatorname{re}_k(\ell, c)-\operatorname{ro}_k(\ell, c\right))\frac{t_\ell^{c}}{c!}={\rm exp}\left(\sum_{g \in G_k(\ell)}(-1)^{\ell g+1}\frac{\ell^{g-1}}{g}t_\ell^g\right).
\]
\end{lemma}
\begin{proof}
Let $\sigma$ be any permutation of cycle type $(\ell)^c$ and let $S$ be the set of all disjoint cycles in $\sigma$. Let $G_k(\ell)=\{g_1, \dots, g_m\}$, with $g_1 < \dots < g_m$. By Proposition~\ref{coro-sol}, $\sigma$ has $k$th roots if and only if the equation
	\[
	g_1x_1+\dots +g_mx_m=c
	\]
has non-negative integer solutions, where a solution $(p_1, \dots, p_m)$ of previous equation means that $\sigma$ has $k$th roots of cycle type $(g_1\ell)^{p_1} \dots (g_m\ell)^{p_m}$. We can obtain all these roots by running over all the weak ordered partitions $(A_1, \dots, A_m)$ of $S$. Indeed, if $(A_1, \dots, A_m)$ is such a partition, the number of $k$th of roots associated to this partition is given by $f_{k, \ell, g_1}(|A_1|) \cdots f_{k, \ell, g_m}(|A_m|)$, where this product is different from $0$ if $|A_i|$ is a multiple of $g_i$, for every $i$.   Let $\mathcal{A}$ be the set of all weak ordered partitions of $S$ into $m$ blocks. The number of $k$th roots of $\sigma$ is equal to 
	\[
	\sum_{(A_1, \dots, A_m) \in \mathcal{A}} f_{k, \ell, g_1}(|A_1|) \cdots f_{k, \ell, g_m}(|A_m|).
	\]
Now, for a given partition $(A_1, \dots, A_m)$ with 
	\[
	f_{k, \ell, g_1}(|A_1|) \cdots f_{k, \ell, g_m}(|A_m|) \neq 0,
	\]
the sign of 
	\[
(-1)^{|A_1|/g_1(\ell g_1+1)}f_{k, \ell, g_1}(|A_1|) \cdots (-1)^{|A_m|/g_m(\ell g_m+1)}f_{k, \ell, g_m}(|A_m|),
	\]
determine the parity of the $k$th roots of $\sigma$ of cycle type 
	$(g_1\ell)^{p_1} \dots (g_m\ell)^{p_m}$, where $p_i=|A_i|/g_i$. Therefore, the number $\operatorname{re}_k\left(\ell, c\right)-\operatorname{ro}_k\left(\ell, c\right)$ is equal to 
	\[
	\sum_{(A_1, \dots, A_m) \in \mathcal{A}} (-1)^{|A_1|/g_1(\ell g_1+1)}f_{k, \ell, g_1}(|A_1|) \cdots (-1)^{|A_m|/g_m(\ell g_m+1)}f_{k, \ell, g_m}(|A_m|), 
	\] 
and the desired exponential generating function is obtained by Proposition 5.1.3 in Stanley's book \cite{stanley} and Proposition~\ref{propgenmenos}. 
\end{proof}

 Let $\operatorname{re}_k({\bf c})$ (resp. $\operatorname{ro}_k({\bf c})$) denote the number of even (resp. odd) $k$th roots of a permutation of cycle type ${\bf c}$.  The following multivariable exponential generating function, in the variables $t_1, t_2, \dots $, for the difference between the number of even $k$th roots and the number of odd $k$th roots of permutations of any cycle type follows from Lemmas \ref{dife-gene} and \ref{fungen-lfijo}.

\begin{lemma}\label{egfdifference} Let $n, k$ be a positive integers and let $c_1, \dots, c_n$ be non-negative integers. For  $n=c_1+2c_2+\dots+nc_n$, the coefficient of $\frac{t_1^{c_1} \dots t_n^{c_n}}{c_1!\dots c_n!}$ in the expansion of  
\[
{\rm exp}\left(\sum_{\ell \geq 1} \sum_{g \in G_k(\ell)}(-1)^{\ell g+1}\frac{\ell^{g-1}}{g}t_\ell^g\right)
\]
is equal to the number $\operatorname{re}_k({\bf c})-\operatorname{ro}_k({\bf c})$, with ${\bf c}=(c_1, \dots, c_n)$. 
\end{lemma}

\begin{proof}[Proof of Theorem~\ref{maintheorem}]
Let $r_k(\si)$ denote the number of $k$th roots of permutation $\sigma$. We have that  
\[
2\operatorname{re}_k(\sigma)=\operatorname{re}_k(\si)+\operatorname{ro}_k(\si)+\operatorname{re}_k(\si)-\operatorname{ro}_k(\si)=r_k(\sigma)+(\operatorname{re}_k(\si)-\operatorname{ro}_k(\si)).
\]
Similarly $2\operatorname{ro}_k(\sigma)=r_k(\sigma)-(\operatorname{re}_k(\si)-\operatorname{ro}_k(\si))$. Therefore, the result follows immediately from Theorem~\ref{theolmr} and Lemma~\ref{egfdifference}.
\end{proof}

\section{Particular cases}\label{cases}
If $k$ is odd, then any solution of the equation $x^k=\si$ should have the same parity as $\si$, so the generating function is the same as the one given in Theorem~\ref{theolmr}. Therefore, in this section $k$ is a fixed even integer. 

In some examples, we will use, without an explicit mention, the following observation.

\begin{obs}\label{obs1}
Let $k$ be an even integer. If $\ell$ is even, then $G_k(\ell)$ is a set of even integers.
\end{obs}

\subsection*{Permutations of cycle type $(\ell)^c$}

For a fixed positive integer $\ell$, we have that 
\begin{equation}\label{eq1}
\sum_{c\geq 0} \operatorname{re}_k(\ell, c)\frac{t^{c}}{c!}=\frac{1}{2}\exp \left( \sum_{g \in G_k(\ell)} \frac{\ell^{g-1}}{g}t^g \right)+\frac{1}{2}{\rm exp}\left(\sum_{g \in G_k(\ell)}(-1)^{\ell g+1}\frac{\ell^{g-1}}{g}t^g\right),
\end{equation}
and 
\begin{equation}\label{eq2}
\sum_{c\geq 0} \operatorname{ro}_k(\ell, c)\frac{t^{c}}{c!}=\frac{1}{2}\exp \left( \sum_{g \in G_k(\ell)} \frac{\ell^{g-1}}{g}t^g \right)-\frac{1}{2}{\rm exp}\left(\sum_{g \in G_k(\ell)}(-1)^{\ell g+1}\frac{\ell^{g-1}}{g}t^g\right).
\end{equation}

With these expressions, we can obtain the generating functions of the following sequences in OEIS: A000704, A061129, A061130, A061131,  A061132, A001465, A061136 and A061137. For example, sequence A061131 corresponds to the number of even $8$th roots of the identity permutation. In this case $\ell=1$ and $G_8(1)=\{1, 2, 4, 8\}$. Therefore
\begin{eqnarray*}
\sum_{c\geq 0} re_8(1, c)\frac{t^{c}}{c!}&=&\frac{1}{2}\exp\left(t + \frac{1}{2}t^2 + \frac{1}{4}t^4 + \frac{1}{8}t^8\right)+\frac{1}{2}\exp\left(t - \frac{1}{2}t^2 - \frac{1}{4}t^4 - \frac{1}{8}t^8\right)\\
&=&{\rm exp}\left(t\right){\rm cosh}\left(\frac{1}{2}t^2 + \frac{1}{4}t^4 + \frac{1}{8}t^8\right).
\end{eqnarray*}


We can make further simplifications of equations~(\ref{eq1}) and (\ref{eq2}).  First, we consider the case when $\ell$ is even.  By Observation~\ref{obs1} we have that

\begin{eqnarray*}
\sum_{c\geq 0} \operatorname{re}_k(\ell, c)\frac{t^{c}}{c!}
&=&{\rm cosh}\left( \sum_{g \in G_k(\ell)} \frac{\ell^{g-1}}{g}t^g \right)
\end{eqnarray*}
and
\begin{eqnarray*}
\sum_{c\geq 0} \operatorname{ro}_k(\ell, c)\frac{t^{c}}{c!}&=&{\rm sinh}\left( \sum_{g \in G_k(\ell)} \frac{\ell^{g-1}}{g}t^g \right).
\end{eqnarray*}

For $\ell$ odd, let $GO_k(\ell)=\{g \in G_k(\ell) \colon g \text{ is odd }\}$ and let $GE_k(\ell)=G_k(\ell)-GO_k(\ell)$. Then
\begin{eqnarray*}
\sum_{c\geq 0} \operatorname{re}_k(\ell, c)\frac{t^{c}}{c!}&=&\frac{1}{2}\exp \left( \sum_{g \in G_k(\ell)} \frac{\ell^{g-1}}{g}t^g \right)+\frac{1}{2}{\rm exp}\left(\sum_{g \in GO_k(\ell)}\frac{\ell^{g-1}}{g}t^g-\sum_{g \in GE_k(\ell)}\frac{\ell^{g-1}}{g}t^g\right),\\
&=&\exp \left( \sum_{g \in GO_k(\ell)} \frac{\ell^{g-1}}{g}t^g \right) {\rm cosh}\left( \sum_{g \in GE_k(\ell)}  \frac{\ell^{g-1}}{g}t^g \right).
\end{eqnarray*}
Similarly, for the case of odd $k$th roots we have
\begin{eqnarray*}
\sum_{c\geq 0} \operatorname{ro}_k(\ell, c)\frac{t^{c}}{c!}&=&{\rm exp}\left(\sum_{g \in GO_k(\ell)} \frac{\ell^{g-1}}{g}t^g \right) {\rm sinh}\left( \sum_{g \in GE_k(1)}  \frac{\ell^{g-1}}{g}t^g \right).
\end{eqnarray*}

For the case of the identity permutation ($\ell=1$) we have that $G_k(1)  =\left\{m\colon m|k\right\}$. Therefore,  
\begin{eqnarray*}
\sum_{c\geq 0} \operatorname{re}_k(1, c)\frac{t^{c}}{c!}&=&{\rm exp}\left(\sum_{\substack{g|k\\g \text{ odd }}} \frac{1}{g}t^g \right) {\rm cosh}\left( \sum_{\substack{g|k\\g \text{ even }}}  \frac{1}{g}t^g \right),
\end{eqnarray*}
and 
\begin{eqnarray*}
\sum_{c\geq 0} \operatorname{ro}_k(1, c)\frac{t^{c}}{c!}
&=&{\rm exp}\left(\sum_{\substack{g|k\\g \text{ odd }}} \frac{1}{g}t^g \right) {\rm sinh}\left( \sum_{\substack{g|k\\g \text{ even }}}  \frac{1}{g}t^g \right).
\end{eqnarray*}

In particular, for the case $k=2^m$, we have 
\begin{eqnarray*}
\sum_{c\geq 0} re_{2^m}(1, c)\frac{x^{c}}{c!}&=&\frac{1}{2}\exp\left(\sum_{i=0}^m \frac{1}{2^i}x^{2^i}\right)+\frac{1}{2}\exp\left(x-\sum_{i=1}^m \frac{1}{2^i}x^{2^i}\right)\\
&=&{\rm exp}(x){\rm cosh}\left(  \frac{1}{2}x^2 +\dots+ \frac{1}{2^m}x^{2^m} \right).
\end{eqnarray*}
This generating function was used in the work of Koda, Sato and Tskegahara~\cite{koda}. For the case of odd roots we have
\[
\sum_{c\geq 0} ro_{2^m}(1, c)\frac{x^{c}}{c!}={\rm exp}(x){\rm sinh}\left(  \frac{1}{2}x^2 +\dots+ \frac{1}{2^m}x^{2^m} \right).
\]
\subsection{Square roots of permutations}
For the case of even square roots we have the following consequence of Theorem~\ref{maintheorem}. 
\begin{corollary}\label{fg-raices-cua-pares} The coefficient of $t_1^{c_1} \dots t_n^{c_n} /(c_1!\dots c_n!)$ in the expansion of
\[
\prod_{j\geq 1}{\rm exp}\left(t_{2j-1}\right){\rm cosh}\left(\sum_{j \geq 1} \left (\frac{2j-1}{2}t_{2j-1}^2+jt_{2j}^2 \right) \right)
\]
is the number of even square roots of a permutation of cycle type ${\bf c}=(c_1, \dots, c_n)$, and  in the expansion of
\[
\prod_{j\geq 1}{\rm exp}\left(t_{2j-1}\right){\rm sinh}\left(\sum_{j \geq 1} \left (\frac{2j-1}{2}t_{2j-1}^2+jt_{2j}^2 \right) \right)
\]
is the number of odd square roots of a permutation of cycle type ${\bf c}$.
\end{corollary}
\begin{proof}
We rewrite Theorem~\ref{maintheorem} for the case of even square roots. When $k=2$, $G_2(\ell) \subseteq \{1, 2\}$. We have two cases depending of the parity of $\ell$. If $\ell=2j-1$, with $j \in \nn$, then $G_2(2j-1) =\{1, 2\}$. Thus	
	\[
	\sum_{g \in G_k(\ell)} (-1)^{\ell g+1}\frac{\ell^{g-1}}{g}t^g_{\ell}=t_{2j-1}-\frac{2j-1}{2}t^2_{2j-1} 
	\]
and
	\[
\sum_{g \in G_k(\ell)} \frac{\ell^{g-1}}{g}t^g_{\ell}=t_{2j-1}+\frac{2j-1}{2}t^2_{2j-1}.
\]
If $\ell=2j$, with $j \in \nn$, then $G_2(2j) =\{2\}$. Therefore  
	\[
	\sum_{g \in G_k(\ell)} (-1)^{\ell g+1}
	\frac{\ell^{g-1}}{g}t^g_{\ell}=- jt^2_{2j}
	\]
and 
	\[
\sum_{g \in G_k(\ell)} \frac{\ell^{g-1}}{g}t^g_{\ell}=jt^2_{2j}.
\]
Therefore, the exponential generating in Theorem~\ref{maintheorem} becomes
{\small
\[
	\frac{1}{2}\left(\exp  \left(  \sum_{j \geq 1} \left(t_{2j-1}+\frac{2j-1}{2}t^2_{2j-1}+jt^2_{2j} \right)\right)+\exp  \left(  \sum_{j \geq 1} \left(t_{2j-1}-\frac{2j-1}{2}t^2_{2j-1}-jt^2_{2j} \right)\right)\right).
	\]
	}
From which we obtain
	{\small 
	\[
\frac{1}{2}\prod_{j\geq 1} {\rm exp}\left(t_{2j-1}\right)\left( \prod_{j\geq 1}{\rm exp} \left( \frac{2j-1}{2}t_{2j-1}^2+jt_{2j}^2\right)+ \prod_{j\geq 1}{\rm exp} \left( -\frac{2j-1}{2}t_{2j-1}^2-jt_{2j}^2 \right)\right),
\]
}
that is equal to
\[
\prod_{j\geq 1}{\rm exp}\left(t_{2j-1}\right){\rm cosh}\left(\sum_{j \geq 1} \left (\frac{2j-1}{2}t_{2j-1}^2+jt_{2j}^2 \right) \right).
\]
The proof for the case of odd square roots is similar.
	\end{proof}


\section*{Acknowledgments}

L.M.R. was partially supported by PROFOCIE grant 2018-2021, through UAZ-CA-169.


\begin{thebibliography}{100}

\bibitem{annin}
S. Annin, T. Jansen and C. Smith, On $k$th roots in the symmetric and alternating groups, {\it Pi Mu Epsilon J.} {\bf 12}(10) (2009), 581--589.

\bibitem{bona}
M. B\'{o}na, A. McLennan and D. White, Permutations with roots, {\it Random Structures \& Algorithms} {\bf 17}(2) (2000), 157--167.

\bibitem{boucher}
I. Z. Bouwer and W. W. Chernoff,
\newblock Solutions to $x^r=\alpha$ in the symmetric group, 
\newblock {\em Ars. Combin.} {\bf 20} (1985) 83--88.



\bibitem{cher2} W. W. Chernoff, Solutions to $x^r=\alpha$ in the alternating group, Twelfth British Combinatorial Conference (Norwich 1989), {\em Ars Combin.} {\bf 29}(C) (1990), 226--227.


\bibitem{cher} W. W. Chernoff, Permutations with $p^l$th roots, 13th British Combinatorial Conference (Guildford, 1991), {\it Discrete Math.} {\bf 125} (1994), 123--127.


\bibitem{chigi} N. Chigira, The solutions of $x^d=1$ in finite groups, {\it J. Algebra} {\bf 180}(3) (1996), 653--661.

\bibitem{chs}
S. Chowla, I. N. Herstein and W. R. Scott, The solution of $x^d=1$ in symmetric groups, {\it Norske Vid. Selsk. Forh., Trondheim} {\bf 25} (1952), 29--31.


\bibitem{ghs}
A. Groch, D. Hofheinz and R. Steinwandt, A practical attack on the root problem in braid groups, {\it Contemp. Math.} {\bf 418} (2006), 121--132.

\bibitem{koda} T. Koda, M. Sato and Y. Takegahara, $2$-adic properties for the numbers of involutions in the alternating groups, {\it J. Algebra Appl.} {\bf 14}(4) (2015), 1550052, 21 pp.

\bibitem{lmr} J. Lea\~nos, R. Moreno and L. Rivera-Mart\'inez, On the number of $m$th roots of permutations, {\em Australas. J. Combin.} {\bf 52} (2012), 41--54.

\bibitem{lucidopour}
M. S. Lucido and M. R. Pournaki, Elements with square roots in finite groups, {\it Algebra Colloq.} {\bf 12}(4) (2005), 677--690.

\bibitem{lucidopour2}
M. S. Lucido and M. R. Pournaki, Probability that an element of a finite group has a square root, {\it Colloq. Math.} {\bf 112} (2008), 147--155.

\bibitem{mw}
 L. Moser and M. Wyman, On solutions of $x^d = 1$ in symmetric groups, {\it Canadian J. Math.} {\bf 7} (1955), 159--168.

\bibitem{niem} A. C. Niemeyer and C. E. Praeger, On permutations of order dividing a given integer {\it J. Algebraic Combin.} {\bf 26} (2007), 125--142.

\bibitem{oeis} OEIS Foundation Inc. (2020), The On-Line Encyclopedia of Integer Sequences, http://oeis.org. 

\bibitem{pavlov1}
A. I. Pavlov,  The number of solutions of the equation $x^k=a$ in the symmetric group $S_n$, {\it Mat. Sb. (N. S.)} {\bf 112}(154)(1980), 380--395; English {\em transl. Math. USSR Sb.} {\bf 40} (1981).

\bibitem{pourn}
M. R. Pournaki, On the number of even permutations with roots, {\it Australas. J.Combin.} {\bf 45} (2009), 37--42.

\bibitem{pouyanne}
N. Pouyanne, On the number of permutations admitting an $m$-th root, {\em Electron. J. Combin.} {\bf 9} (2002), $\#R3.$, 12 pp.

\bibitem{roichman} Y. Roichman, A note on the number of $k$-roots in $S_n$, {\em S\'em. Lothar. Combin.} {\bf 70} (2013), Art. 
B70i, 5pp.

\bibitem{sade} A. Sadeghieh and K. Ahmadidelir, $n$-th roots in finite polyhedral and centro-polyhedral groups, {\em Proc. Indian Acad. Sci. Math Sci.} {\bf 125}(4) (2015), 487--499. 

\bibitem{sade2} A. Sadeghieh and H. Dostie, The $n$th roots of elements in finite groups, {\em Mathematical Sciences} {\bf 2} (2008), 347--356.

\bibitem{stanley} R. P. Stanley, Enumerative Combinatorics Vol 2, Cambridge Univ. Press, 1999.

\bibitem{york} C. W. York, Enumerating $k$th roots in the symmetric inverse monoid, {\em J. Combin. Math. Combin. Comput.} {\bf 108}  (2019), 147--159.
\\

L. G., UNIVERSIDAD AUT\'ONOMA DE SAN LUIS POTOS\'I, MEXICO.\\  E-mail: glebsky@cactus.iico.uaslp.mx.\\

M. L. and L. M. R., UNIVERSIDAD AUT\'ONOMA DE ZACATECAS, MEXICO.\\
E-mails: mliiqon@gmail.com and luismanuel.rivera@gmail.com.\\
\end{thebibliography}
\end{document}